\documentclass[10pt]{article}

\usepackage[T1]{fontenc}
\usepackage[utf8]{inputenc}
\usepackage{lmodern}
\usepackage[margin=1in]{geometry}
\usepackage{amsmath,amssymb,amsthm,mathtools}
\usepackage{graphicx}
\usepackage{xcolor}
\usepackage{booktabs}
\usepackage{caption}
\usepackage{float}
\usepackage{tikz}
\usepackage{pgfplots}
\usepackage[hidelinks]{hyperref}

\pgfplotsset{compat=1.18}
\usepgfplotslibrary{groupplots}
\usetikzlibrary{arrows.meta,calc,intersections}

\theoremstyle{plain}
\newtheorem{Theorem}{Theorem}[section]
\newtheorem{Lemma}[Theorem]{Lemma}

\theoremstyle{definition}
\newtheorem{Definition}{Definition}[section]

\newtheorem{Remark}{Remark}[section]

\newcommand{\E}{\mathbb{E}}
\newcommand{\R}{\mathbb{R}}
\newcommand{\Pp}{\mathbb{P}}

\title{\textbf{The Gittins Index for a Transparent One-Armed\\
Bandit with a Continuous Payoff Spectrum\\
in Equilibrium States}}

\author{%
Marcin Makowski\textsuperscript{1,*},
Edward W. Piotrowski\textsuperscript{2},
and Jan L. Cieśliński\textsuperscript{1}\\[0.6em]
\small\textsuperscript{1}Faculty of Physics, Department of Mathematical Methods in Physics,\\
\small University of Białystok, ul.~Ciołkowskiego 1L, 15-245 Białystok, Poland\\
\small\textsuperscript{2}Independent Researcher, Białystok, Poland\\[0.4em]
\small\textsuperscript{*}Corresponding author: \texttt{m.makowski@uwb.edu.pl}\\
\small E-mail addresses: \texttt{j.cieslinski@uwb.edu.pl};
\texttt{qmgames@gmail.com}\\[0.4em]
\small
}
\date{}

\begin{document}
\maketitle

\begin{abstract}
We present an analogue of the classical Gittins index for a one-armed decision problem with a continuous spectrum of payoffs. The model assumes that the decision-maker observes independent realizations of a random variable and, at each step, decides whether to accept the current opportunity or continue observing. We show that the optimal strategy takes the form of a threshold rule, while the corresponding reservation index is determined by a one-dimensional fixed-point equation with a direct decision-theoretic interpretation. The model is illustrated with an example of bookmaker betting related to horse racing and the Kelly criterion. This perspective allows the proposed index to be viewed as a threshold of informational advantage. This, in turn, points to potential applications in optimal stopping problems and decision-making under uncertainty.
\end{abstract}


\vspace{1em}
\section{Introduction: The Gittins Index}

The multi-armed bandit problem is a classical problem in probability theory and machine learning, concerning sequential choice among competing sources of random payoffs. At each step, the decision-maker selects one of the arms, observes the outcome, and updates the subsequent strategy on the basis of this information. The central challenge in this class of models is to find a rule that orders the arms according to their current attractiveness in the simplest possible way. To this end, a one-dimensional index is assigned to each arm. In the classical Bayesian setting, the seminal result of John C. Gittins shows that, for an infinite time horizon and exponential discounting, the strategy of selecting the arm with the largest index is optimal \cite{LattimoreSzepesvari2020,Gittins1979}.

What is the index assigned to an arm in a given state? Its classical definition follows from reducing the multi-armed problem to an auxiliary one-armed stopping problem. It is assumed that the state of a single arm evolves as a Markov process \((S_t)_{t\ge 1}\) on a state space \(S\), with transition kernel \(P(x,\cdot)\), where the initial state is \(x\in S\). Each state is associated with an immediate reward \(r(S_t)\), and future payoffs are discounted by a factor \(\alpha\in(0,1)\). If in each round the player may either stop the process or continue activating the arm, then the value of the auxiliary problem can be written as
\[
v_\gamma(x)
:=
\sup_{\tau\in\mathcal T}
\mathbb E_x\!\left[
\sum_{t=1}^{\tau-1}\alpha^{t-1}\bigl(r(S_t)-\gamma\bigr)
\right],
\]
where \(\mathcal T\) denotes the set of stopping times with respect to the natural filtration of the process, taking values in \(\{1,2,\dots\}\). The parameter \(\gamma\) plays the role of a continuation cost or an outside value, whereas the choice \(\tau=1\) corresponds to immediate stopping, and hence to an empty sum and a value equal to \(0\). The Gittins index of state \(x\) is defined as the smallest value of the charge \(\gamma\) at which the decision-maker becomes indifferent between immediate withdrawal and further play, that is,
\[
g(x):=\inf\{\gamma\in\mathbb R:\ v_\gamma(x)=0\}.
\]
Equivalently, this index can be written as
\begin{equation}\label{anG}
g(x)=
\sup_{\tau\ge 2}
\frac{\mathbb E_x\!\left[\sum_{t=1}^{\tau-1}\alpha^{t-1}r(S_t)\right]}
{\mathbb E_x\!\left[\sum_{t=1}^{\tau-1}\alpha^{t-1}\right]}.
\end{equation}
The second form shows that the Gittins index is the maximal discounted average reward attainable by an optimal choice of the stopping time in a one-armed problem \cite{LattimoreSzepesvari2020}.

It is worth explicitly explaining all the elements that make up the above definition. By \(x\in S\) we denote the fixed current state of the arm from which the analysis of the auxiliary stopping problem begins. This does not mean that the state of the arm remains equal to \(x\) throughout. On the contrary, the subsequent evolution is described by the process \((S_t)_{t\ge 1}\), with \(x\) serving only as the initial state, i.e. \(S_1=x\). The Gittins index \(g(x)\) is therefore the value assigned to an arm that is currently in state \(x\). 
In Bayesian models, \(x\) usually takes the form of a concise summary of the observations made so far, containing all information relevant for determining the current posterior distribution. The transition kernel \(P\) of the Markov process describing the evolution of the arm's state defines the probability distribution of the state to which the arm moves after the next activation.

The function \(r\) determines the one-step payoff from activating the arm in a given state. This means that if at time \(t\) the arm is in state \(S_t=x\), then the current activation is associated with the reward \(r(x)\), which enters the sum of discounted payoffs. It is usually assumed that rewards are random. Then, if \(Y_t\) denotes the actual random reward received at the \(t\)-th activation of the arm, and before this activation the arm was in state \(S_t=x\), the function \(r\) assigns to this state the expected value of that reward, namely
\[
r(x)=\mathbb E[Y_t\mid S_t=x].
\]

The parameter \(\alpha\) plays the role of a discount factor in the model. The smaller its value, the greater the weight assigned to current payoffs and the smaller the weight assigned to benefits deferred in time. By \(\tau\) we denote a stopping time, that is, a rule deciding how long it is worth continuing the exploration of a single arm. The parameter \(\gamma\) plays the role of the price of further play with a given arm. It can be viewed as a cost incurred at each decision to continue. The Gittins index is then the threshold price that is still worth paying for the possibility of further play in state \(x\).

In the quotient representation of the Gittins index, the numerator
\[
\mathbb E_x\!\left[\sum_{t=1}^{\tau-1}\alpha^{t-1}r(S_t)\right]
\]
denotes the expected discounted total reward that can be obtained from the present moment until stopping, if the process starts in state \(x\). In turn, the denominator
\[
\mathbb E_x\!\left[\sum_{t=1}^{\tau-1}\alpha^{t-1}\right]
\]
describes the expected discounted length of continuation, that is, the discounted number of future activations of the arm until stopping. The quotient of these two quantities can therefore be interpreted as the expected discounted average reward from further activation of the arm, calculated from the perspective of the current state \(x\).

In the multi-armed model, each arm \(i\) is assigned its current index \(g(S_i(t))\), and then the arm with the largest index is selected. In the classical discounted Bayesian model, such a rule is an optimal policy. As a result, the complex comparison of entire future random paths is reduced to a comparison of real numbers assigned to individual arms \cite{LattimoreSzepesvari2020,Gittins1979}.

\noindent
The Gittins index itself, as well as all kinds of its contemporary analogues and generalizations, are intensively studied tools in broadly understood decision theory. Recent studies have developed methods for computing and approximating the index
\cite{Baas2026,Noba2026} and explored its applications in scheduling and
queueing theory \cite{Ram2026,Har2025}. Index ideas also appear in Bayesian optimization, including problems of costly exploration and stopping rules \cite{Xie2024,Xie2025}, as well as in reinforcement learning \cite{Dhan2024}. Contemporary analogues of the Gittins index are also used in graph exploration \cite{Choo2025}, security games \cite{Ling2026}, platform models \cite{Seo2026}, and search problems with travel costs \cite{Mel2026}. The paper by Scully and Terenin \cite{Scu2025} is of a survey nature, and in it the authors slightly change the perspective on the Gittins index. It is not merely a difficult dynamic programming equation, but a universal principle for designing decision algorithms. 

It is worth noting the broad range of potential applications of methods
based on Gittins' idea. They have been applied in IoT systems, where they
are used, among others, for scheduling data transmission with unmanned
aerial vehicles (UAVs) \cite{JinWanQinZhu2025}, client selection in
so-called decentralized federated learning \cite{CheZhaWanWan2025},
as well as for reducing the costs of distributed machine learning
\cite{EggBitWacGun2025}. Similar ideas can also be found in other areas
of scientific research, such as adaptive clinical trial design
\cite{SmiVil2018}, the theory of reservation prices in optimal stopping
problems \cite{Sei1992}, knowledge-gradient methods, and preference
learning \cite{Fra2009,WuGar2026}. Taken together, these examples provide
a strong motivation for studying the multi-armed bandit problem.

Although the proposed value is referred to as a ``Gittins index analogue'', we must explicitly highlight a fundamental difference from the classical framework. Because our model relies on independent and identically distributed observations without state evolution, the classical exploration--exploitation dilemma is absent. The analogy to Gittins' work lies strictly in the structural methodology of the solution: as in the classical approach, we reduce the multi-armed selection process to an auxiliary one-armed optimal stopping problem. We define the index as a reservation threshold representing the point of indifference, and demonstrate that the multi-armed choice simplifies to picking the arm with the highest index. In this sense, our approach adopts the index-based decision architecture of Gittins' original idea, adapting it to a transparent, stationary setting.

\section{Motivation and Outline of the Paper}

The aim of this paper is to propose a transparent analogue of the classical Gittins index approach for a one-armed decision problem with a continuous spectrum of payoffs. This article refers to the intuition underlying the multi-armed bandit problem, but formulates it in a more transparent model. Unlike the standard Markov model with discounted reward, we consider a model based on a sequence of independent observations and the criterion of average payoff per cycle. These observations may be interpreted in various ways. Since in the article we primarily refer to a market analogy, we shall interpret them as the potential net profit from a transaction, the rate of return from an opportunity, or, more broadly, the payoff associated with a given opportunity. Acceptance means realizing this opportunity and obtaining the payoff, whereas rejection means moving on to the next observation. After entering into a transaction, there is a deterministic time equal to \(1\), interpreted as the time of execution, holding, or settlement of the position. This assumption reflects a natural market constraint on capital liquidity: funds committed to one transaction cannot be used immediately to exploit another opportunity, and the execution and settlement of the position itself require time.

Analogously to the classical Gittins idea, the attractiveness of a given investment opportunity is assessed by analyzing an auxiliary one-armed problem, that is, a problem associated with a single source of offers. The number obtained in this way is then used to compare different sources of market opportunities. As in the case of the Gittins index, this number has the interpretation of a boundary value, the so-called reservation threshold. It is the value of the current investment opportunity at which the investor regards immediate acceptance and continued waiting as equally beneficial. In other words, this threshold determines the boundary between a situation in which it is still worthwhile to wait for a better opportunity and a situation in which the available transaction should already be realized. In this sense, it determines the minimal attractiveness of an opportunity at which entering into the transaction becomes optimal.

The structure of the problem under consideration appears in a broad class of
sequential decisions. Its characteristic elements are the random arrival of
successive opportunities, an immediate and irreversible decision to accept or
reject the current opportunity, the time cost associated with further search,
and a period during which the resource is unavailable after an offer is
accepted. In such situations, the threshold $r^\star$ determines the minimum
value of the current opportunity that compensates for giving up further
search.

A particularly natural interpretation comes from optimal foraging theory. A
predator or consumer encounters successive potential food sources with random
energy values. Rejecting an item means continuing the search, while accepting
it yields energy and requires additional time for capture, preparation, or
consumption. Classical models of diet choice and attack strategies maximize
the long-run rate of energy intake and lead to rules determining which
encountered opportunities should be accepted
\cite{MacArthurPianka1966,Charnov1976,FantinoPreston1988}. In this
interpretation, the symbol $X_n$ used later in the paper may denote the net
value of the $n$-th encountered food source, and the optimal threshold
$r^\star$ is the minimum value that justifies its acceptance.

A similar structure appears in psychological studies of sequential choice and
search. Unlike standard experiments in which all options are presented
simultaneously, in foraging tasks successive opportunities become available
one at a time. The decision-maker must then compare the current benefit with
the opportunity cost of the time required for further search. Studies of
learning in such tasks indicate that the long-run average reward rate may
serve as an adaptive reference level \cite{ConstantinoDaw2015}. The threshold
rule is also related to the concept of satisficing, according to which search
ends when the first opportunity exceeding a fixed aspiration level is found
\cite{Simon1955}. In the model presented here, however, the aspiration level
is not chosen arbitrarily, but results from the solution of the fixed-point
equation and maximizes the average reward per unit time.

From an algorithmic perspective, the model is related to sequential selection
problems. In these problems, values are revealed successively and the
decision-maker must immediately accept or reject each one. Threshold rules
play an important role, among other areas, in prophet inequalities and posted
price mechanisms \cite{CorreaEtAl2021,CorreaEtAl2022}. Unlike fixed-horizon
models, the process considered here renews after each acceptance, and the
optimality criterion is the average reward per unit time. The model can
therefore also be interpreted as a simplified admission control problem:
successively arriving tasks have random values, a rejected task does not
occupy the resource, while an accepted task blocks it for a fixed service
time.

The construction presented here is not intended to reproduce the full
complexity of each of the applications mentioned above, but should be treated
as a deliberately simplified baseline model. Temporal dependence,
non-stationarity, random service time, switching costs, and dependence between
successive opportunities lead to natural extensions in which the acceptance
threshold may depend on time, state, or the history of the process.

The importance of simple reference models of this kind lies in isolating the
fundamental decision mechanism and determining its consequences under
conditions that allow rigorous analytical results to be obtained. This makes
it possible to distinguish properties resulting from the structure of the
sequential choice problem itself from effects caused by additional
assumptions, such as non-stationarity, dependence between observations, or
variable service time. An exact solution of a simple model also provides a
benchmark for evaluating more complex models, numerical algorithms, and
learning procedures, and yields transparent hypotheses that can subsequently
be investigated empirically.

The main motivation of the paper is to examine, within a simplified and
analytically solvable model, the fundamental mechanism connecting the value of
the current opportunity with the cost of further waiting and the service time
following its acceptance. In the following part, we show that the resulting
decision problem can be reduced to a single optimal reservation threshold. An
exact solution of such a model provides a reference point for more complex
extensions, threshold learning procedures, and comparisons of different
sources of opportunities.

The remainder of the article is organized as follows. In Section~\ref{opt}, we introduce the one-armed stopping model and determine the optimal threshold rule together with the corresponding reservation threshold. Next, in Section~\ref{secW}, we move to the multi-armed index model and show that the selection problem reduces to comparing the reservation indices of the individual arms. In Section~\ref{ucz}, we discuss learning the reservation threshold when the distribution of observations is unknown. In Section~\ref{kel12}, we present an illustration of the model using the example of a bookmaker bet based on the Kelly criterion, and in Section~\ref{dyw}, we show the connection of this example with the Kullback--Leibler divergence and the informational interpretation of the threshold rule. The article concludes with Section~\ref{last}, which contains final remarks and possible directions for further research.

\section{Optimal stopping rule}\label{opt}

The idea behind the construction below refers to the works \cite{Pio,Mak}, devoted to the maximization of profit intensity. We show that the problem of an optimal investment decision can be formulated as an optimal stopping problem for an appropriately constructed stochastic process. However, the optimal rule presented here is general and does not depend on the adopted interpretation.

Consider a sequence $\{X_n\}_{n \ge 1}$ of real-valued, independent random variables with the same distribution (i.i.d.), consistent with the distribution of a random variable $X$. At successive times $n = 1, 2, \ldots$, we observe the realization $X_n$ and make one of two decisions:
\begin{itemize}
    \item \textbf{accept} (stop) -- we take the current value as the reward,
    \item \textbf{reject} (continue) -- we give up the current value and move on to the next observation.
\end{itemize}
The decision is made on the basis of the currently available information. Each observation consumes exactly $1$ unit of time. When we decide to accept the value $X_n$, we receive a reward equal to $X_n$, and then incur a deterministic service-time cost equal to $1$ unit. Then the entire cycle renews from the beginning.

We define the length of the cycle as $L:=\tau+1$, and the reward as $R:=X_\tau$.
The average payoff criterion takes the form
\begin{equation}\label{sre}
g(\tau):=\frac{\mathbb E[R]}{\mathbb E[L]}=\frac{\mathbb E[X_\tau]}{\mathbb E[\tau+1]}.
\end{equation}
The aim is to maximize the above value, analogously to the Gittins problem in (\ref{anG}). 
For $r\in\mathbb R$, a threshold rule is the stopping time
\[
\tau_r:=\inf\{n\ge 1:\ X_n\ge r\}.
\]

\noindent
In what follows, we write $g(r):=g(\tau_r)$, i.e. the value of the average payoff criterion for the threshold rule with threshold $r$.

\begin{Remark}
\noindent \textbf{Assumption on service time.} After the stopping time $\tau$ (acceptance), there is a deterministic service period, which we normalize to $1$ unit of time. This means that the system becomes unavailable for the duration of this activity, and only after its completion is it possible to return to the observation process. This assumption models the so-called dead time or refractory period of the system, a necessary time cost associated with finalizing the decision (e.g. executing an order, physically serving a customer), during which the system cannot process new requests.
\end{Remark}
\subsection{Auxiliary Results}
\vspace{0.2cm}
\noindent
Let
\begin{equation} \label{ozn1}
(X-r)^+ := \max\{X-r,\,0\} = (X-r)\,\mathbf 1_{\{X\ge r\}}.
\end{equation} 
In particular,
\[
X^+:=(X-0)^+=\max\{X,0\}\,.
\]
The quantity $\E[(X-r)^+]$ is the expected excess of the variable $X$ above the threshold $r$.

\begin{Lemma}\label{L1}
Let $\E[X^+]<\infty$. There exists exactly one real number $r^\star$ satisfying the equation
\begin{equation}\label{roww1}
\E[(X-r^\star)^+]-r^\star=0
\end{equation}
and it is nonnegative, i.e. $r^\star\in[0,\infty)$. 
\end{Lemma}

\begin{proof}
For $r\in\R$, define the function
\[
\phi(r):=\E[(X-r)^+]-r.
\]
Since
\[
(X-r)^+ \le X^+ + (-r)^+,
\]
the assumption $\E[X^+]<\infty$ implies that $\E[(X-r)^+]<\infty$ for every $r$, and hence $\phi(r)$ takes finite values
and $\phi(r)\in\R$.

For arbitrary $r,s\in\R$, the inequality
\[
\big|(X-r)^+-(X-s)^+\big|
\le \big|(X-r)-(X-s)\big|
=|r-s|
\]
holds.
We have
\[
|\phi(r)-\phi(s)|
= \left| \big(\E[(X-r)^+]-\E[(X-s)^+]\big) -(r-s)\right|
\]
and, by the triangle inequality,
\[
|\phi(r)-\phi(s)|
\le \left|\E[(X-r)^+]-\E[(X-s)^+]\right|+|r-s|
\le |r-s|+|r-s|
=2|r-s|.
\]
Thus $\phi$ is continuous. If $s>r$, then $(X-s)^+\le (X-r)^+$, and therefore
\[
\phi(s)-\phi(r)=\E[(X-s)^+-(X-r)^+]-(s-r)\le -(s-r)<0,
\]
so $\phi$ is strictly decreasing. 

We now show that $\phi$ has a unique zero. If $\E[X^+]=0$, then
\[
\phi(0)=\E[X^+]=0,
\]
so $r^\star=0$. 

Assume now that $\E[X^+]>0$. Then
\[
\phi(0)=\E[X^+]>0.
\]
Moreover, for $r\ge 0$,
\[
0\le (X-r)^+\le X^+,
\]
and hence
\[
\phi(r)=\E[(X-r)^+]-r\le \E[X^+]-r \xrightarrow[r\to+\infty]{}-\infty.
\]
Since $\phi$ is continuous, the intermediate value theorem yields some
$r^\star>0$ such that $\phi(r^\star)=0$. Since $\phi$ is strictly decreasing,
this zero is unique.

Finally, the equality $\phi(r^\star)=0$ is equivalent to
\[
r^\star=\E[(X-r^\star)^+],
\]
so in particular $r^\star\ge 0$.
\end{proof}

\vspace{0.4cm}
\noindent

We will show that the optimization of the average payoff (\ref{sre}) within the class of threshold
rules reduces to the maximization (with respect to the threshold $r$) of the one-dimensional function $\operatorname{rap}(r)$ :
$$\operatorname{rap}(r):= \frac{N(r)}{1+S(r)}\,,$$
where, for $r \in \R$,
\begin{equation}\label{eq:notations_r}
\begin{aligned}
S(r) &:= \Pp(X\ge r),\qquad
N(r) := \E\!\big[X\,\mathbf 1_{\{X\ge r\}}\big].\qquad
\end{aligned}
\end{equation}

By an equilibrium state, we mean a stationary threshold rule \(\tau_r\)
for which the value of the reservation threshold \(r\) agrees with the
average payoff per cycle obtained by applying the same rule, that is,
\[
\operatorname{rap}(r)=r.
\]
Thus, this term refers to a fixed point of the threshold mapping and not
to a stationary distribution of the learning process.

\begin{Lemma}\label{lem:geom}
Let $\{X_n\}_{n\ge 1}$ be independent random variables with the same distribution as the random variable $X$, and let $r\in\R$ be such that $S(r)>0$ and $\E[X^+]<\infty$.
Then the time $\tau_r:=\inf\{n\ge 1:\ X_n\ge r\}$ is finite, $\E[\tau_r]<\infty$, and
\[
g(\tau_r)=\frac{\E[X_{\tau_r}]}{\E[\tau_r+1]}=\frac{N(r)}{1+S(r)}=\operatorname{rap}(r).
\]
\end{Lemma}

\begin{proof}
Set $p:=S(r)\in(0,1]$ and $q:=1-p=\Pp(X<r)\in[0,1)$.

Using the assumption that the random variables $X_i$ are independent, for $n\ge 0$ we have
\[
\Pp(\tau_r>n)=\Pp(X_1<r,\dots,X_n<r)=q^n.
\]
Thus the random variable $\tau_r$ has the geometric distribution on $\{1,2,\dots\}$ with parameter $p$:
\[
\Pp(\tau_r=k)=q^{k-1}p,\qquad k\ge 1.
\]
In particular, $\Pp(\tau_r<\infty)=1$.Using the formula for the sum of a geometric series, we obtain
\begin{equation} \label{skon}
\E[\tau_r]=\sum_{n=0}^\infty \Pp(\tau_r>n)=\sum_{n=0}^\infty q^n=\frac1p,
\qquad
\E[\tau_r+1]=1+\E[\tau_r]=1+\frac1p=\frac{1+p}{p}.
\end{equation}
Observe that the random variable $X_{\tau_r}$ is integrable. 
Since the events $\{\tau_r=k\}$, $k\ge 1$, are pairwise disjoint and
$\Pp(\tau_r<\infty)=1$, we have
\[
X_{\tau_r}^+
=
\sum_{k=1}^\infty X_k^+\mathbf 1_{\{\tau_r=k\}}\,.
\]
Hence, using the monotone convergence theorem, which allows us to interchange the expectation of a nonnegative increasing limit with the limit of the expectations, we obtain
\[
\E[X_{\tau_r}^+]
=
\sum_{k=1}^\infty
\E\big[X_k^+\mathbf 1_{\{\tau_r=k\}}\big]
=
\sum_{k=1}^\infty
q^{k-1}\E\big[X^+\mathbf 1_{\{X\ge r\}}\big]
=
\frac{\E[X^+\mathbf 1_{\{X\ge r\}}]}{p}
<\infty.
\]
By $Y^-:=\max\{-Y,0\}$, denote the negative part of the random variable $Y$. On the event $\{\tau_r<\infty\}$, we have $X_{\tau_r}\ge r$, and therefore
\[
X_{\tau_r}^- \le (-r)^+.
\]
Hence
\[
\E[X_{\tau_r}^-]\le (-r)^+<\infty.
\]
Thus $\E[|X_{\tau_r}|]<\infty$.

From the definition of $\tau_r$, it follows that
\[
\{\tau_r=k\}=\{X_1<r,\dots,X_{k-1}<r,\ X_k\ge r\},
\]
and hence
\[
\mathbf 1_{\{\tau_r=k\}}
=\mathbf 1_{\{X_k\ge r\}}\prod_{j=1}^{k-1}\mathbf 1_{\{X_j<r\}}.
\]
Consequently,
\[
\E[X_{\tau_r}]
=\sum_{k=1}^\infty \E\Big[X_k\,\mathbf 1_{\{X_k\ge r\}}\prod_{j=1}^{k-1}\mathbf 1_{\{X_j<r\}}\Big].
\]
By the independence of the variables $X_1,\dots,X_k$ and the assumption that they are identically distributed, we obtain:
\[
\E\Big[X_k\,\mathbf 1_{\{X_k\ge r\}}\prod_{j=1}^{k-1}\mathbf 1_{\{X_j<r\}}\Big]
=\E\big[X\,\mathbf 1_{\{X\ge r\}}\big]\cdot \Big(\Pp(X<r)\Big)^{k-1}
= N(r)\,q^{k-1}.
\]
Thus
\[
\E[X_{\tau_r}]
=\sum_{k=1}^\infty N(r)\,q^{k-1}
=\frac{N(r)}{1-q}
=\frac{N(r)}{S(r)}.
\]
Substituting the above into the definition (\ref{sre}), we get
\[
g(\tau_r)=\frac{\E[X_{\tau_r}]}{\E[\tau_r+1]}
=\frac{N(r)/S(r)}{(1+S(r))/S(r)}
=\frac{N(r)}{1+S(r)}
=\operatorname{rap}(r).
\]
\end{proof}

\subsection{Optimal Reservation Threshold}
\noindent
\vspace{0.3cm}
The functional
\begin{equation}\label{fun}
J_\lambda(r):=\E\big[X_{\tau_r}-\lambda(\tau_r+1)\big]
\end{equation}
is the standard transformation of a quotient problem into a difference form,
cf. the classical Dinkelbach method for fractional programming
\cite{Din,Scha}. In our case, for every threshold \(r\) such that
\(S(r)>0\), we have
\begin{equation}\label{eq:J_identity}
J_\lambda(r)
=
\E[\tau_r+1]\left(
\frac{\E[X_{\tau_r}]}{\E[\tau_r+1]}-\lambda
\right)
=
\E[\tau_r+1]\bigl(g(\tau_r)-\lambda\bigr)= \E[\tau_r+1]\bigl(\operatorname{rap}(r)-\lambda\bigr).
\end{equation}
Since \(\E[\tau_r+1]>0\), the sign of \(J_\lambda(r)\) is the same as the sign of
$g(\tau_r)-\lambda=\operatorname{rap}(r)-\lambda$. 
In the proof of the theorem below, we will apply this observation for
\(\lambda=r^\star\), where \(r^\star\) is the unique solution of the equation
\[
\E[(X-r)^+]-r=0.
\]
We will show that the solution of the above equation determines the optimal reservation threshold within the class of threshold rules.
\begin{Theorem}
Let $\{X_n\}_{n\ge 1}$ be i.i.d. with the same distribution as $X$, let $\E[X^+]<\infty$, and let $S(0)>0$.
Denote by $r^\star\in[0,\infty)$ the unique solution of the equation
\[
\E[(X-r^\star)^+]-r^\star=0.
\]
Then
\begin{equation}\label{eq:sup_rap}
\sup_{r:\,S(r)>0}\operatorname{rap}(r)
=r^\star,
\end{equation}
and
\begin{equation}\label{eq:fixed_point}
\operatorname{rap}(r^\star)=r^\star.
\end{equation}
Therefore, \(r^\star\) determines the unique equilibrium state.
\end{Theorem}

\begin{proof}
Let \(r\in\mathbb R\) be such that \(S(r)>0\). By Lemma~\ref{lem:geom}, we have
\[
\E[X_{\tau_r}]=\frac{N(r)}{S(r)},
\qquad
\E[\tau_r+1]=\frac{1+S(r)}{S(r)}.
\]
Hence
\[
J_\lambda(r)
=
\frac{N(r)-\lambda S(r)-\lambda}{S(r)}
=
\frac{\E[(X-\lambda)\mathbf 1_{\{X\ge r\}}]-\lambda}{S(r)}.
\]
Substituting \(\lambda=r^\star\), we obtain
\[
J_{r^\star}(r)
=
\frac{\E[(X-r^\star)\mathbf 1_{\{X\ge r\}}]-r^\star}{S(r)}.
\]
Since
\[
(X-r^\star)\mathbf 1_{\{X\ge r\}}
\le
(X-r^\star)^+,
\]
it follows from the definition of \(r^\star\) that
\[
J_{r^\star}(r)
\le
\frac{\E[(X-r^\star)^+]-r^\star}{S(r)}
=0.
\]
By \eqref{eq:J_identity}, we therefore obtain
\[
\operatorname{rap}(r)\le r^\star
\qquad
\text{for every } r \text{ such that } S(r)>0.
\]
Thus
\[
\sup_{r:\,S(r)>0}\operatorname{rap}(r)\le r^\star.
\]

It remains to verify that the threshold \(r^\star\) is admissible, i.e. $S(r^\star)>0$, and attains this value.
If \(S(r^\star)=0\), then \((X-r^\star)^+=0\), and hence the equation
\[
\E[(X-r^\star)^+]-r^\star=0
\]
would imply \(r^\star=0\). But then
$S(r^\star)=S(0)>0$,
which gives a contradiction. 

For \(r=r^\star\), we have
\[
J_{r^\star}(r^\star)
=
\frac{\E[(X-r^\star)\mathbf 1_{\{X\ge r^\star\}}]-r^\star}{S(r^\star)}
=
\frac{\E[(X-r^\star)^+]-r^\star}{S(r^\star)}
=0.
\]
Using \eqref{eq:J_identity} again, we obtain
\[
\operatorname{rap}(r^\star)=r^\star.
\]
Consequently,
\[
\sup_{r:\,S(r)>0}\operatorname{rap}(r)
=
r^\star.
\]
\end{proof}

\begin{Remark}
The assumption \(S(0)>0\) serves only to exclude the degenerate case
and ensures that the threshold \(r^\star\) is admissible, i.e. \(S(r^\star)>0\).
\end{Remark}

\begin{Remark}[One-armed interpretation]
The obtained optimal threshold $r^\star$ can be interpreted as the reservation value
of a single source of offers. The rule
\[
\tau_{r^\star}=\inf\{n\ge 1:\ X_n\ge r^\star\}
\]
compares the current offer \(X_n\) with the reservation threshold \(r^\star\). This threshold determines the boundary value
at which accepting the offer and continuing to wait are equivalent. This follows from the equation
\[
\E[(X-r^\star)^+]=r^\star,
\]
which shows that the expected excess attainable as a result of further waiting is equal to the reservation threshold. The one-armed solution presented above
constitutes an elementary building block for the multi-armed model,
in which each arm has its own reservation value.
\end{Remark}

It is worth emphasizing that the threshold rule associated with
$r^\star$ is optimal within the class of all admissible stopping
times, and not only within the class of threshold rules.

By an admissible stopping time, we mean a stopping time $\tau$ with
respect to the natural filtration
\[
\mathcal F_n=\sigma(X_1,\dots,X_n),
\]
taking values in $\{1,2,\ldots\}\cup\{\infty\}$ and satisfying
\[
\E[\tau]<\infty.
\]
This condition already implies that
\[
\Pp(\tau<\infty)=1,
\]
so no separate almost-sure finiteness assumption is required. The
condition $\E[\tau]<\infty$ is natural in the present
renewal-reward setting because it guarantees a finite expected cycle
length, which appears in the denominator of the average-payoff
criterion. Allowing stopping times with infinite expected cycle
length would require a different asymptotic performance criterion
and is outside the scope of the present model.

Let $\tau$ be any admissible stopping time. We first show that the
positive part of the stopped payoff relative to $r^\star$ is
integrable. Since $\Pp(\tau<\infty)=1$, Tonelli's theorem gives
\[
\E[(X_\tau-r^\star)^+]
=
\sum_{n\geq 1}
\E\!\left[
    (X_n-r^\star)^+
    \mathbf 1_{\{\tau=n\}}
\right].
\]
Because
\[
\{\tau=n\}\subseteq\{\tau\geq n\},
\]
we obtain
\[
\E[(X_\tau-r^\star)^+]
\leq
\sum_{n\geq 1}
\E\!\left[
    (X_n-r^\star)^+
    \mathbf 1_{\{\tau\geq n\}}
\right].
\]
The event $\{\tau\geq n\}$ belongs to $\mathcal F_{n-1}$ and is
therefore determined before $X_n$ is observed. Since the observations
are independent and identically distributed, the event
$\{\tau\geq n\}$ is independent of $X_n$, and $X_n$ has the same
distribution as $X$. Consequently,
\[
\begin{aligned}
\E[(X_\tau-r^\star)^+]
&\leq
\sum_{n\geq 1}
\E[(X_n-r^\star)^+]\Pp(\tau\geq n) \\
&=
\E[(X-r^\star)^+]
\sum_{n\geq 1}\Pp(\tau\geq n) \\
&=
r^\star\E[\tau]
<\infty,
\end{aligned}
\]
where we used the defining equation
\[
\E[(X-r^\star)^+]=r^\star
\]
and the tail-sum identity
\[
\sum_{n\geq 1}\Pp(\tau\geq n)=\E[\tau].
\]

It follows that the positive part of $X_\tau-r^\star$ has finite
expectation. Hence $\E[X_\tau-r^\star]$ is well defined as an
extended real number in $[-\infty,\infty)$. In particular, neither
the value $+\infty$ nor an indeterminate expression of the form
$\infty-\infty$ can occur. Therefore,
\begin{equation}\label{osza}
\E[X_\tau-r^\star]
\leq
\E[(X_\tau-r^\star)^+]
\leq
r^\star\E[\tau].
\end{equation}
Thus,
\[
\E[X_\tau]
\leq
r^\star\bigl(\E[\tau]+1\bigr)
=
r^\star\E[\tau+1].
\]

If $\E[X_\tau]=-\infty$, we use the convention
\[
g(\tau)=-\infty.
\]
Otherwise, the average payoff is defined in the usual way. Since
\[
0<\E[\tau+1]<\infty,
\]
we obtain in both cases
\[
g(\tau)
=
\frac{\E[X_\tau]}{\E[\tau+1]}
\leq
r^\star.
\]

Finally, consider the threshold stopping time
\[
\tau_{r^\star}
=
\inf\{n\geq 1:\ X_n\geq r^\star\}.
\]
Since $S(r^\star)>0$, this stopping time has a geometric distribution
with success probability $S(r^\star)$, and therefore
\[
\E[\tau_{r^\star}]
=
\frac{1}{S(r^\star)}
<\infty.
\]
Thus, $\tau_{r^\star}$ is admissible. Moreover, it attains the value
\[
g(\tau_{r^\star})=r^\star.
\]
Consequently, $\tau_{r^\star}$ is optimal within the class of all
admissible stopping times, and the optimal value of the stopping
problem is equal to $r^\star$.

The argument also covers heavy-tailed payoff distributions, provided
that their positive part satisfies the standing assumption
$\E[X^+]<\infty$. The negative part of the stopped payoff may have
infinite expectation, in which case the average payoff is assigned
the value $-\infty$.

\begin{Remark}
More generally, the model may allow for a sequence of nonnegative
random service times having the same distribution as a random
variable $Y$, independent of the observation sequence, with finite
mean
\[
\mu=\mathbb{E}[Y]\in(0,\infty).
\]
The service time is assumed not to be known before the acceptance
decision. In this case, the length of a cycle is $\tau+Y$, and the
average-payoff criterion becomes
\[
g_Y(\tau)
=
\frac{\mathbb{E}[X_\tau]}
     {\mathbb{E}[\tau+Y]}
=
\frac{\mathbb{E}[X_\tau]}
     {\mathbb{E}[\tau]+\mu}.
\]

For the threshold rule $\tau_r$, this gives
\[
\operatorname{rap}_{\mu}(r)
=
\frac{N(r)}{1+\mu S(r)}.
\]
The optimal threshold $r_\mu^\star$ is the unique nonnegative
solution of
\[
r_\mu^\star
=
\mu\,\mathbb{E}\bigl[(X-r_\mu^\star)^+\bigr],
\]
and the corresponding optimal average payoff is
\[
\sup_{r:\,S(r)>0}\operatorname{rap}_{\mu}(r)
=
\operatorname{rap}_{\mu}(r_\mu^\star)
=
\frac{r_\mu^\star}{\mu}.
\]
Thus, the threshold structure of the solution is preserved, although
for $\mu\neq1$ the optimal threshold is no longer numerically equal
to the optimal average payoff. For $\mu=1$, we recover the equations
used in the main model.

This extension relies on the assumption that the service time is
independent of the accepted payoff and is not observed before the
acceptance decision. If the service time depends on the accepted
opportunity or is known in advance, the optimal rule may depend on
both the payoff and the service time and requires a separate analysis.
\end{Remark}

To illustrate the dependence of the optimal reservation threshold on the
reward distribution, Table~\ref{tab:selected-thresholds} presents the
corresponding values for five standard distributions: uniform, exponential,
gamma, Weibull, and normal. The parameters of all distributions were chosen
so that their expected value was equal to~$1$. In each case, the optimal
threshold is determined by the fixed-point equation
\[
\mathbb{E}\bigl[(X-r^\star)^+\bigr]=r^\star.
\]

\begin{center}
\begin{minipage}{0.98\linewidth}
\centering

\captionof{table}{
Optimal reservation thresholds for selected reward distributions with
expected value equal to~$1$. The symbols $\varphi$ and $\Phi$ denote the
probability density function and cumulative distribution function of the
standard normal distribution, respectively.
}
\label{tab:selected-thresholds}

\small
\renewcommand{\arraystretch}{1.35}

\begin{tabular}{
    @{}
    p{0.38\linewidth}
    p{0.42\linewidth}
    c
    @{}
}
\toprule
\textbf{Reward distribution}
&
\textbf{Equation determining \(r^\star\)}
&
\(\boldsymbol{r^\star}\)
\\
\midrule

Uniform distribution
\(\mathcal{U}(0,2)\)
&
\(\displaystyle
\frac{(2-r^\star)^2}{4}=r^\star
\)
&
\(0.535898\)
\\[1ex]

Exponential distribution
with rate parameter \(1\)
&
\(\displaystyle
e^{-r^\star}=r^\star
\)
&
\(0.567143\)
\\[1ex]

Gamma distribution
with shape parameter \(2\) and rate parameter \(2\)
&
\(\displaystyle
(r^\star+1)e^{-2r^\star}=r^\star
\)
&
\(0.530045\)
\\[1ex]

Weibull distribution
with shape parameter \(2\) and scale parameter
\(\displaystyle 2/\sqrt{\pi}\)
&
\(\displaystyle
\operatorname{erfc}
\left(
\frac{\sqrt{\pi}}{2}r^\star
\right)
=r^\star
\)
&
\(0.517005\)
\\[1ex]

Normal distribution
with expected value \(1\) and variance \(1\),
\(\mathcal{N}(1,1)\)
&
\(\displaystyle
\varphi(r^\star-1)
+(1-r^\star)\Phi(1-r^\star)
=r^\star
\)
&
\(0.618425\)
\\

\bottomrule
\end{tabular}
\end{minipage}
\end{center}

Although all five distributions have the same expected reward, their
corresponding optimal reservation thresholds differ. Among the cases
considered, the largest threshold value is obtained for the normal
distribution with variance equal to~$1$. This shows that the optimal
threshold depends on the entire reward distribution, including its
dispersion, support, and tail behavior, not only on its expected value.

\section{Multi-armed index model}\label{secW}
The multi-armed bandit model can be interpreted as a simplified market model in which
individual arms correspond to different sources of investment opportunities. An arm may be
a specific asset, trading strategy, market sector, asset class, or a particular
method of capital allocation. Activating a given arm leads to the observation of the next value
generated by this source. This is the value of the current investment opportunity, for
example, the potential net profit from a transaction, the rate of return, or a more broadly understood payoff
associated with a given opportunity. Accepting this value means realizing the investment opportunity,
whereas rejecting it means giving up this possibility and continuing the search.
Since a rejected observation does not generate a payoff, but only lengthens the cycle time, the
multi-armed problem can be interpreted as a selection process among competing market
opportunities. In this interpretation, the index of each arm measures the attractiveness of a given source
of investment opportunities relative to the value of further waiting.

We consider a finite set of arms
\[
\mathcal I=\{1,\dots,m\}.
\]
Each arm $i\in\mathcal I$ corresponds to an i.i.d. sequence
\[
X^{(i)}_1,X^{(i)}_2,\dots
\]
with the same distribution as the random variable $X^{(i)}$. At a decision time, the investor selects
one arm and observes one offer from this arm. After the observation, the investor may:
\begin{enumerate}
    \item accept the offer and end the cycle with a reward equal to the observed value,
    \item reject the offer and, at the next time, choose any arm again.
\end{enumerate}
After acceptance, there is a deterministic service time equal to $1$.

\begin{Definition}
For a fixed arm $i\in\mathcal I$, we define its reservation index by
\[
r_i^\star \in [0,\infty),
\qquad
\E\big[(X^{(i)}-r_i^\star)^+\big]=r_i^\star,
\]
provided that $\E[(X^{(i)})^+]<\infty$ and $\Pp(X^{(i)}\ge 0)>0$.
\end{Definition}

\begin{Remark}
By the one-armed analysis, for each arm $i$ the number $r_i^\star$ exists,
is uniquely determined, and satisfies
\[
r_i^\star=\sup_{r:\,S_i(r)>0}\operatorname{rap}_i(r),
\qquad
\operatorname{rap}_i(r):=\frac{N_i(r)}{1+S_i(r)},
\]
where
\[
S_i(r):=\Pp(X^{(i)}\ge r),
\qquad
N_i(r):=\E\big[X^{(i)}\mathbf 1_{\{X^{(i)}\ge r\}}\big].
\]
\end{Remark}

\begin{Definition}
Let
\[
r_{\max}:=\max_{i\in\mathcal I} r_i^\star.
\]
An arm $i^\star\in\mathcal I$ is called an index-optimal arm if
\[
r_{i^\star}^\star=r_{\max}.
\]
\end{Definition}

In the described model, the reservation indices \(r_i^\star\) can be treated as
values that order the arms according to their attractiveness. 
If, after each rejection, any arm can be freely chosen, there are no
switching costs or learning effects, and the observations generated by different
arms are independent, then the natural choice is to activate the arm with
the largest reservation index.

For each arm \(i\), we have \(r_{\max}\ge r_i^\star\), and therefore
\[
\E\big[(X^{(i)}-r_{\max})^+\big]\le r_{\max}.
\]
This means that no arm gives an expected excess above the level
\(r_{\max}\) that would justify choosing an arm with a lower index. On the other
hand, every arm \(i^\star\) satisfying
\[
r_{i^\star}^\star=r_{\max}
\]
attains the value \(r_{\max}\) if it is continuously activated and the threshold rule
\[
\tau^{(i^\star)}_{r_{\max}}
=
\inf\{n\ge 1:\ X_n^{(i^\star)}\ge r_{\max}\}
\]
is applied to it.
In this sense, the multi-armed selection problem reduces to comparing the reservation indices
\(r_1^\star,\dots,r_m^\star\). Since no arm gives an expected excess above the level $r_{\max}$ greater than the cost of further waiting, mixing or switching between arms cannot improve the value attained by the arm with the largest index.

\section{Learning the reservation threshold with an unknown distribution}\label{ucz}

\subsection{General scheme for learning the reservation threshold}

In the preceding considerations, we assumed that the distribution of the random variable \(X\) is known, which makes it possible to determine the optimal threshold \(r^\star\) as the solution of a fixed-point equation. In market practice, however, this assumption is too strong. The investor observes only successive realizations of investment opportunities and, on their basis, must gradually learn the value of a single arm. In the proposed model, this learning consists in estimating the one-dimensional decision parameter, namely the reservation threshold $r$, by a trial-and-adjustment method. In practice, this proceeds as follows: the investor adopts a certain threshold $r_k$, applies the corresponding decision rule for some time, and then checks what average payoff per cycle has been achieved in this way. If the obtained value is greater than the adopted threshold, the threshold should be increased, whereas if it is smaller, the threshold should be decreased. Repeating this procedure leads to a gradual approximation of the threshold value consistent with the observed data.

For a fixed threshold \(r_k\), we run the threshold rule
\[
\tau_{r_k}:=\inf\{n\ge 1:\ X_n\ge r_k\}
\]
for \(K_k\) consecutive cycles. Let
\[
R_{k,1},\dots,R_{k,K_k}
\]
denote the payoffs obtained in these cycles, while
\[
L_{k,1},\dots,L_{k,K_k}
\]
denote the corresponding cycle lengths. A natural estimator of the average payoff per cycle for the threshold \(r_k\) is the quotient
\[
\widehat g_k
:=
\frac{\sum_{j=1}^{K_k} R_{k,j}}
{\sum_{j=1}^{K_k} L_{k,j}}.
\]

The estimator $\widehat g_k$ is a ratio estimator and, for a finite number of
cycles $K_k$, is generally biased because the expectation of a ratio need not
equal the ratio of expectations. For a fixed threshold $r_k$, however,
successive cycles are independent and identically distributed. The strong law
of large numbers therefore implies that $\widehat g_k$ is a consistent
estimator and converges to the true average payoff per cycle $g(r_k)$ as
$K_k\to\infty$. Under additional moment assumptions, its bias is typically of
order
$\mathcal{O}(1/K_k)$ \cite{Cochran1977}.

If $K_k$ remains fixed, the bias need not vanish and the algorithm may
converge to a value different from the true threshold $r^\star$. A decreasing
learning step reduces random fluctuations, but does not by itself remove
persistent bias. Therefore, to obtain convergence to the true threshold, the
number of cycles $K_k$ should be increased appropriately as learning
progresses. Under the standard assumptions used in the analysis of stochastic
approximation algorithms, including appropriately decreasing learning steps,
control of the moments of the estimation error, stability of the iterations,
and stability of the fixed point, the vanishing of the bias as $K_k$ increases
removes one of the obstacles to convergence of the algorithm to the true
threshold $r^\star$ \cite{KushnerYin2003}. A complete proof of convergence
would, however, require a detailed verification of these conditions for the
procedure under consideration and lies beyond the scope of this paper.

We update the threshold according to the iteration
\begin{equation}\label{eq:learning_update}
r_{k+1}=r_k+\alpha_k\bigl(\widehat g_k-r_k\bigr),
\end{equation}
where \((\alpha_k)_{k\ge 0}\subset(0,1]\) is a sequence of learning step sizes. The coefficient \(\alpha_k\) determines how strongly the new estimate \(\widehat g_k\) affects the change of the threshold: for \(\alpha_k=1\), we obtain the full fixed-point iteration, whereas for \(0<\alpha_k<1\), the update has an averaged character, which may increase the stability of the procedure in stochastic settings.

The interpretation of the update \eqref{eq:learning_update} is direct. If
\[
\widehat g_k>r_k,
\]
then the current threshold is too low relative to the estimated value of the arm, and therefore the next iteration
increases the threshold. If, on the other hand,
\[
\widehat g_k<r_k,
\]
then the threshold is too high and is decreased. At the fixed point, when the estimated average payoff
per cycle agrees with the current threshold, the update no longer changes the value of the parameter.

A special case of \eqref{eq:learning_update} is the full fixed-point iteration, corresponding to the choice
\[
\alpha_k=1,
\]
for which we obtain the simple scheme
\[
r_{k+1}=\widehat g_k.
\]
Then each new threshold value is directly equal to the estimated average payoff per cycle
obtained at the previous threshold. This scheme constitutes a natural learning procedure from data:
the investor tests the current threshold, observes the payoffs and realization times generated by it,
and then adjusts the decision parameter to the empirically observed quality of the strategy.

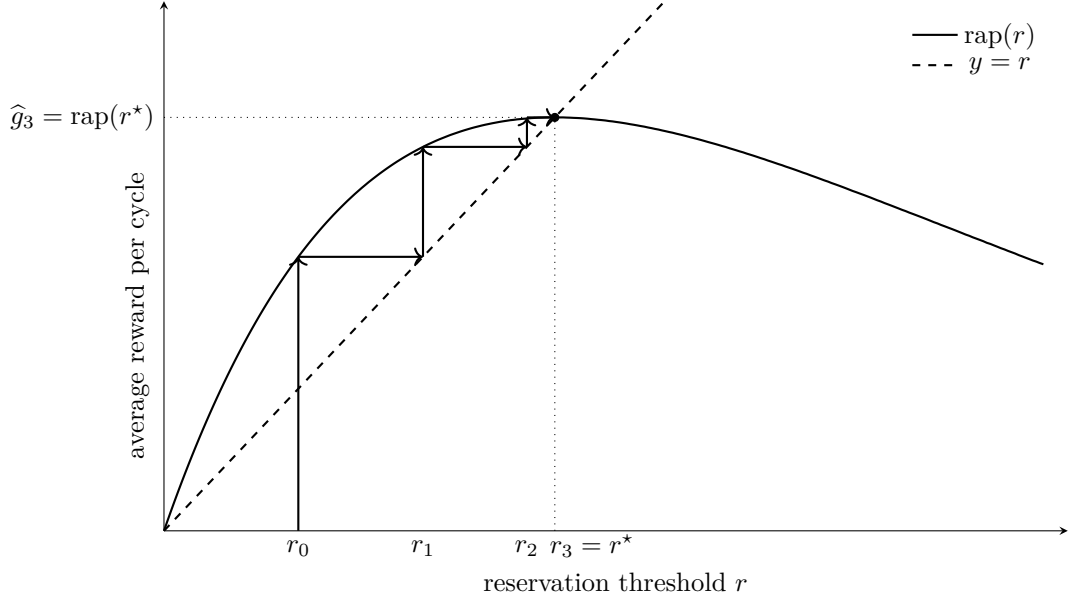
\begin{figure}[H]
\centering

\pgfmathdeclarefunction{rap}{1}{%
  \pgfmathparse{2.718281828*#1*exp(-#1/1.6)}%
}

\pgfmathsetmacro{\rstar}{1.6}

\pgfmathsetmacro{\rzero}{0.55}
\pgfmathsetmacro{\rone}{rap(\rzero)}
\pgfmathsetmacro{\rtwo}{rap(\rone)}

\pgfmathsetmacro{\rthree}{\rstar}

\begin{tikzpicture}
\begin{axis}[
    width=0.82\textwidth,
    height=0.52\textwidth,
    xmin=0, xmax=3.7,
    ymin=0, ymax=2.05,
    axis lines=left,
    xlabel={reservation threshold $r$},
    xlabel style={yshift=-0.35cm},
    ylabel={average reward per cycle},
ylabel style={xshift=-0.65cm},
    xtick=\empty,
    ytick=\empty,
    clip=false,
    legend style={
        draw=none,
        fill=none,
        at={(0.98,0.98)},
        anchor=north east
    }
]

\addplot[
    thick,
    domain=0:3.6,
    samples=250
]
{rap(x)};
\addlegendentry{$\operatorname{rap}(r)$}

\addplot[
    dashed,
    thick,
    domain=0:2.05,
    samples=2
]
{x};
\addlegendentry{$y=r$}

\addplot[
    only marks,
    mark=*,
    mark size=1.5pt
]
coordinates {(\rstar,\rstar)};

\draw[dotted] (axis cs:\rstar,0) -- (axis cs:\rstar,\rstar);
\draw[dotted] (axis cs:0,\rstar) -- (axis cs:\rstar,\rstar);
\node[left] at (axis cs:0,\rstar) {$\widehat g_3=\operatorname{rap}(r^\star)$};


\draw[->, thick] 
    (axis cs:\rzero,0) -- 
    (axis cs:\rzero,\rone);

\draw[->, thick] 
    (axis cs:\rzero,\rone) -- 
    (axis cs:\rone,\rone);

\draw[->, thick] 
    (axis cs:\rone,\rone) -- 
    (axis cs:\rone,\rtwo);

\draw[->, thick] 
    (axis cs:\rone,\rtwo) -- 
    (axis cs:\rtwo,\rtwo);

\draw[->, thick] 
    (axis cs:\rtwo,\rtwo) -- 
    (axis cs:\rtwo,\rthree);

\draw[->, thick] 
    (axis cs:\rtwo,\rthree) -- 
    (axis cs:\rthree,\rthree);

\node[below] at (axis cs:\rzero,0) {$r_0$};
\node[below] at (axis cs:\rone,0) {$r_1$};
\node[below] at (axis cs:\rtwo,0) {$r_2$};
\node[below, xshift=14pt, yshift= 2.2pt] at (axis cs:\rthree,0) {$r_3=r^\star$};



\end{axis}
\end{tikzpicture}

\caption{Example trajectory of learning the reservation threshold. The curve $\operatorname{rap}(r)$ represents the average payoff per cycle for the threshold rule with threshold $r$. Successive steps show the iteration $r_{k+1}=\operatorname{rap}(r_k)$, which leads to the fixed point $r^\star$, satisfying $\operatorname{rap}(r^\star)=r^\star$. The convergence of the iteration itself requires additional regularity assumptions and does not follow solely from the existence of a fixed point.}
\label{fig:learning-reservation-threshold}
\end{figure}
Figure~\ref{fig:learning-reservation-threshold} presents an example
trajectory of the iteration
\[
r_{k+1}=\operatorname{rap}(r_k),
\]
leading to the fixed point \(r^\star\).
Learning in our model is considerably simpler than in general reinforcement learning \cite{PioRL2010}.
Instead of approximating a value function on an extended state space, we estimate a single number,
which simultaneously describes the value of the arm and generates the optimal threshold rule. The convergence of such a procedure, however, requires
additional regularity conditions, in particular control of the estimation error
$\widehat g_k$. This can be done through an increasing number of observed cycles $K_k$
or appropriately decreasing learning step sizes $\alpha_k$.

\subsection{Numerical illustration of the learning process}
\label{subsec:numeryczna-ilustracja-uczenia}

Figure~\ref{fig:threshold-learning} presents the process of learning the
reservation threshold from successively observed rewards following an
exponential distribution with rate parameter $\lambda=1$. We conducted
$100$ independent Monte Carlo simulations, each comprising $5000$ completed
decision cycles. In each simulation, the initial threshold value was set to
\[
r_0=0.
\]
This is a natural choice for the exponential distribution, whose values are
nonnegative. Consequently, the first reward drawn is accepted in the first
cycle.

In the $k$-th cycle, the threshold $r_{k-1}$ was used. Successive reward
values were drawn until a value not smaller than the current threshold
appeared. This value was accepted and denoted by $R_k$. The cycle length
$L_k$ was equal to the number of opportunities considered, including the
accepted opportunity, increased by one additional unit of time associated
with servicing the accepted opportunity.

The experiment used a fully sequential variant of the learning procedure, in
which the threshold was updated after each completed cycle. Let
\[
T_k=\sum_{j=1}^{k}L_j
\]
denote the total duration of the first $k$ cycles. After completion of the
$k$-th cycle, the threshold was updated according to
\[
r_k
=
\frac{\displaystyle\sum_{j=1}^{k}R_j}{T_k}.
\]
Thus, $r_k$ is the ratio of the cumulative reward obtained in the first $k$
cycles to the total duration of these cycles. The calculated value $r_k$ was
then used as the threshold in cycle $k+1$.

The above update can also be written recursively. Since
\[
\sum_{j=1}^{k-1}R_j=T_{k-1}r_{k-1},
\]
we obtain
\[
r_k
=
\frac{T_{k-1}r_{k-1}+R_k}{T_k}
=
r_{k-1}
+
\frac{L_k}{T_k}
\left(
\frac{R_k}{L_k}-r_{k-1}
\right).
\]

This update has a form analogous to the general learning scheme presented
earlier. Due to the indexing adopted here, in which the $k$-th cycle is
conducted using the threshold $r_{k-1}$, the corresponding quantities take
the form
\[
K_{k-1}=1,
\qquad
\widehat g_{k-1}=\frac{R_k}{L_k},
\qquad
\alpha_{k-1}=\frac{L_k}{T_k}.
\]
It is therefore a variant in which the update is performed after each
individual cycle and the learning step is data-dependent. The value
$\alpha_{k-1}$ is random because it depends on the length of the cycle that
has just been completed.

Each of the $100$ simulations generated its own sequence of estimates
\[
r_1,r_2,\ldots,r_{5000}.
\]
Thus, for each cycle number $k$, $100$ values of $r_k$ were obtained, one
from each independent simulation. The solid black line in the figure shows
the mean of these $100$ values. The gray region is bounded by the 10th and
90th percentiles and therefore contains the middle $80\%$ of the threshold
values obtained in the independent simulations. This region is not a
confidence interval, but a graphical illustration of the variability of the
results across individual realizations of the learning procedure. Its
narrowing means that, as the number of completed cycles increases, the
threshold estimates obtained in different simulations become increasingly
similar.

For the exponential distribution with rate parameter $\lambda=1$, the
equation determining the optimal threshold takes the form
\[
e^{-r^\star}=r^\star,
\]
and thus
\[
r^\star=W(1)\approx0.5671,
\]
where $W$ denotes the Lambert function. The horizontal dashed line in the
figure shows this exact value of the optimal threshold.

The experiment illustrates the empirical stabilization of the estimates near
the exact value $r^\star$. It does not, however, constitute a proof of
convergence of the procedure used. An analytical demonstration of convergence
for the variant with a random, data-dependent step $\alpha_k$ would require
the formulation of additional regularity conditions.

\begin{figure}[h!]
    \centering
    \includegraphics[width=0.90\textwidth]{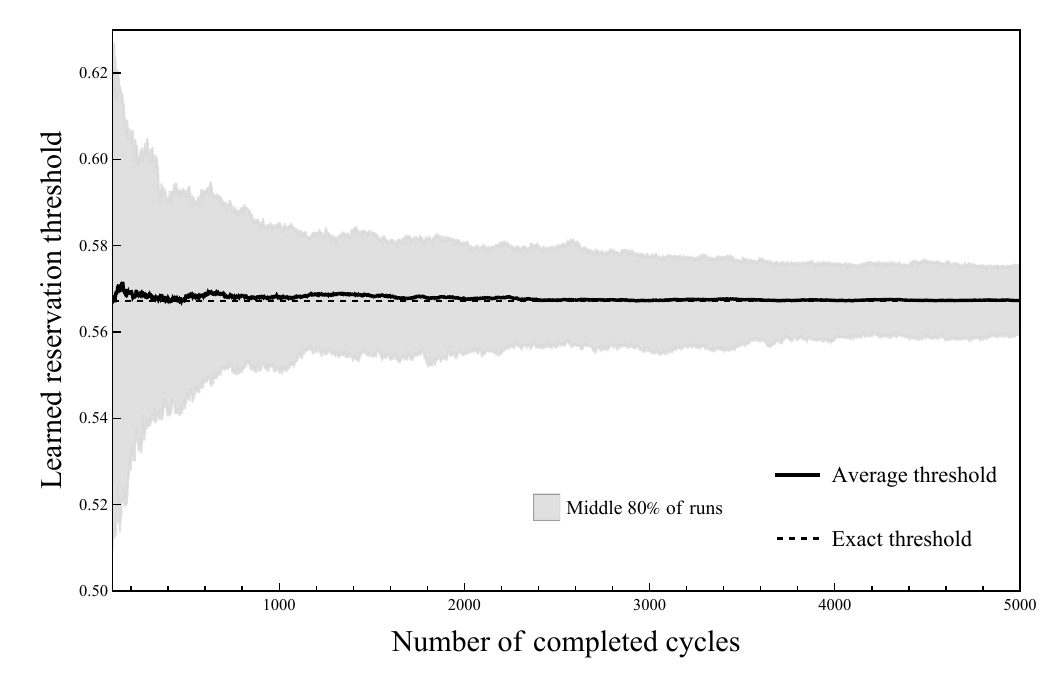}
    \caption{
    Learning the reservation threshold from rewards following an exponential
    distribution with rate parameter $\lambda=1$. Each of the $100$
    independent Monte Carlo simulations comprises $5000$ completed decision
    cycles. The solid black line shows the mean estimated threshold in the
    $100$ simulations. The gray region extends from the 10th to the 90th
    percentile and contains the middle $80\%$ of the values obtained in the
    independent simulations. The horizontal dashed line indicates the exact
    optimal threshold $r^\star=W(1)\approx0.5671$. The narrowing of the gray
    region illustrates the decreasing variation in results across independent
    realizations of the procedure. The plot begins at cycle 100. Values for
    cycles 1 to 99 were omitted because of the high variability of the initial
    estimates.
    }
    \label{fig:threshold-learning}
\end{figure}

\section{An Illustrative Model of a Bookmaker’s Bet}\label{kel12}

As an illustration of the model presented in this paper, consider a sequence of horse races. Each race is interpreted as a single investment opportunity, and its value is identified with the largest attainable logarithmic gain determined by means of the Kelly criterion \cite{Tho2011}. In the \(n\)-th race, the investor considers all available bets and, for each of them, determines the corresponding logarithmic capital increment in accordance with the Kelly criterion. Then \(X_n\) is taken to be the largest of the obtained values. Consequently, the sequence \((X_n)_{n \geq 1}\) can be treated as a sequence of observations in the sense of the one-armed model presented above. Accepting the observation \(X_n\) means placing a bet in the \(n\)-th race on the horse whose local Kelly value attains the maximum. This decision closes the current cycle and ties up capital for an additional unit of time, corresponding to the duration and settlement of the bet. Rejection, on the other hand, means skipping the current race and moving on to the next opportunity. As a result, the investor's decision again takes the form of a threshold rule.

\subsection{Single race and local Kelly value}

Consider the \(n\)-th race, in which \(m_n\) horses take part. For each \(i\in\{1,\dots,m_n\}\), denote by \(p_{n,i}\)
the subjective assessment of the probability that the \(i\)-th horse wins, and by
\(q_{n,i}\in(0,1)\) the share of funds bet on the \(i\)-th horse
in the total betting pool. We assume that
\[
\sum_{i=1}^{m_n} p_{n,i}=1
\qquad\text{and}\qquad
\sum_{i=1}^{m_n} q_{n,i}=1.
\]
We assume that there are no fees or taxes of any kind. This means that 
if the \(i\)-th horse wins, the unit stake placed on
this event is multiplied by \(1/q_{n,i}\).

To connect a single race with the binary Kelly model, for a fixed
horse \(i\) we reduce the situation to two disjoint events:
\[
A_{n,i}:=\{\text{the \(i\)-th horse wins}\},
\qquad
A_{n,-i}:=\{\text{the \(i\)-th horse loses}\}.
\]
The probability of the event \(A_{n,i}\) is \(p_{n,i}\), and that of the event \(A_{n,-i}\) is \(1-p_{n,i}\). Analogously, the market share
corresponding to these two events is equal to \(q_{n,i}\) and
\(1-q_{n,i}\), respectively.

Let \(l\in[0,1)\) denote the fraction of capital staked by the player on the victory
of the \(i\)-th horse. If this horse wins, then the player's normalized final wealth is
\[
1-l+\frac{l}{q_{n,i}}
=
1+\frac{1-q_{n,i}}{q_{n,i}}\,l\,.
\]
In the case of a loss, it is equal to \(1-l\). The expected value of the logarithmic capital increment has the form
\begin{equation}\label{Kelly}
K_{n,i}(l)
:=
p_{n,i}\ln\!\left(1+\frac{1-q_{n,i}}{q_{n,i}}\,l\right)
+
(1-p_{n,i})\ln(1-l).
\end{equation}
By the optimal stake, we shall mean the quantity
\begin{equation}\label{eq:Kelly-stake}
l_{n,i}^{\mathrm K}
:=
\arg\max_{\,0\le l<1} K_{n,i}(l).
\end{equation}
If \(p_{n,i}>q_{n,i}\), then the maximum is attained for
\begin{equation}\label{frr}
l_{n,i}^{\mathrm K}
=
\frac{p_{n,i}-q_{n,i}}{1-q_{n,i}},
\end{equation}
whereas if \(p_{n,i}\le q_{n,i}\), the optimal stake is equal to \(0\). The Kelly value for the \(i\)-th horse in the \(n\)-th race
is defined by
\begin{equation}\label{eq:kappa-ni}
\kappa_{n,i}
:=
\sup_{\,0\le l<1} K_{n,i}(l).
\end{equation}
The above quantity is a measure of the ex ante attractiveness of betting on the victory
of the \(i\)-th horse, measured by the maximal attainable expected
logarithmic capital increment in a single race. If the subjective
probability that the \(i\)-th horse wins is greater than
the probability implied by the market, i.e. \(p_{n,i}>q_{n,i}\), then,
according to the Kelly criterion, one should bet on its victory the fraction of capital
specified by formula \eqref{frr}. Otherwise, one should refrain
from placing the bet. This criterion rules out betting the entire capital: when \(l=1\),
a loss reduces the final wealth to \(0\), and the term \(\ln(1-l)\)
is not defined.

Passing from a single horse to the entire race, we define its
local value by
\begin{equation}\label{eq:Xn-Kelly}
X_n
:=
\max_{1\le i\le m_n}\kappa_{n,i}.
\end{equation}
If the above maximum is attained for more than one horse, then we choose
any index $
i_n^\star \in \arg\max_{1\le i\le m_n}\kappa_{n,i}$.

The i.i.d.\ assumption concerns the sequence of local race values $(X_n)$,
not the individual parameters of each race. The fact that the number of horses
and the subjective and market-implied probabilities vary between races is not
in itself inconsistent with this assumption. These quantities may be treated
as different realizations generated by the same probabilistic mechanism.

In real betting markets, these conditions may not be satisfied. Changes in the
number or structure of races, the method used to determine market
probabilities, the behavior of market participants, and other time-dependent
factors may change the distribution of $X_n$ or introduce dependence between
successive local values. In such cases, the theoretical results derived under
the i.i.d.\ assumption do not apply directly. The model would then have to be
extended to a non-stationary or possibly state-dependent setting in which the
thresholds could depend on time, market conditions, or the observed
characteristics of a race.

Accepting the observation \(X_n\) means placing a bet in the \(n\)-th race
on the horse \(i_n^\star\), using the stake \(l_{n,i_n^\star}^{\mathrm K}\). Rejection means skipping the current race and moving on to
the next opportunity. After acceptance, the capital remains committed for
one further unit of time, corresponding to the duration and settlement of the bet. As a result,
the sequence \((X_n)_{n\ge1}\) can be treated as a sequence of observations in
a one-armed stopping problem, and the Kelly criterion provides
a concrete construction of the local value appearing in the threshold rule.

\subsection{Reservation threshold for local Kelly values}
In the described model, to each race we assign a one-dimensional local value
$$
X_n=\max_{1\le i\le m_n}\kappa_{n,i},
$$
which determines the potential attractiveness of the best available opportunity in a given race.
We assume here that successive races are treated as independent observations, that is, that the conditional probability distribution of future outcomes is independent of what happened in the past.

In this interpretation, the function
\begin{equation}\label{ke11}
\operatorname{rap}(a)=\frac{\E[X\,\mathbf 1_{\{X\ge a\}}]}{1+\Pp(X\ge a)}
\end{equation}
describes the average logarithmic payoff per cycle obtained by a strategy that
accepts only those races for which the local Kelly value satisfies \(X_n\ge a\).

The quantity in (\ref{ke11}) can be interpreted as the quotient of the expected
value of the accepted opportunity and the expected length of the entire cycle
of waiting and settlement of the bet. 
Maximizing \(\operatorname{rap}(a)\) therefore expresses the trade-off between selectivity
and the frequency of placing a bet. The optimal threshold \(r^\star\), determined by the fixed-point equation,
is the reservation value of the local Kelly value. If
the current race has value \(X_n=r^\star\), then the investor is indifferent between
placing the best available bet in this race and continuing to wait
for the next opportunity. Hence,
the investor accepts the race if and only if $X_n\ge r^\star$.
After accepting a race, the investor chooses the horse $
i_n^\star\in\arg\max_{1\le i\le m_n}\kappa_{n,i} $
and bets on its victory the fraction of capital determined by the Kelly criterion. If \(X_n<r^\star\), then the race is skipped.

It is worth emphasizing that the Kelly criterion is not merely a technical rule for determining the stake, but a principle connecting capital growth with the information contained in the structure of the bet. The privileging of the strategy maximizing the expected value of the logarithm of wealth can be interpreted through projective symmetries, as well as in the context of its natural connection with Boltzmann--Shannon entropy and with constraints relevant to large investors \cite{Pio2007}. In this sense, the optimality of the Kelly criterion goes far beyond the question ``how much to stake?'' It is a way of describing the relationship between information, risk, and the long-run rate of capital growth.

\section{Kullback--Leibler divergence and the informational interpretation of the threshold rule}\label{dyw}

In this section, we show that in the bookmaker model based on the Kelly criterion,
the local value of an investment opportunity can be interpreted directly in terms of
the Kullback--Leibler divergence. As a result, the reservation threshold receives
an additional informational interpretation.

\subsection{Kullback--Leibler divergence for a single bet}

Let \(P\) and \(Q\) be discrete probability distributions
defined on the same set of values. We assume that \(P(x)>0\)
implies \(Q(x)>0\). The Kullback--Leibler divergence of the distribution \(P\)
relative to \(Q\) is defined by
\[
D_{\mathrm{KL}}(P\|Q)
:=
\sum_x P(x)\ln\frac{P(x)}{Q(x)},
\]
where the sum runs over all values \(x\) for which the distributions under consideration
are defined. For Bernoulli distributions \(\mathrm{Ber}(p)\) and
\(\mathrm{Ber}(q)\), where \(p,q\in(0,1)\), the above formula takes the following form:
\begin{equation}\label{eq:KL-bernoulli}
D_{\mathrm{KL}}(\mathrm{Ber}(p)\|\mathrm{Ber}(q))
=
p\ln\frac{p}{q}
+
(1-p)\ln\frac{1-p}{1-q}.
\end{equation}

In the horse-racing model, for a fixed race \(n\) and horse \(i\),
we considered the function
\begin{equation} \label{logw}
K_{n,i}(l)
=
p_{n,i}\ln\!\left(1+\frac{1-q_{n,i}}{q_{n,i}}\,l\right)
+
(1-p_{n,i})\ln(1-l),
\qquad 0\le l<1,
\end{equation}
expressing the expected logarithmic capital increment. If \(p_{n,i}>q_{n,i}\), then this function attains its maximum for
\[
l_{n,i}^{\mathrm K}
=
\frac{p_{n,i}-q_{n,i}}{1-q_{n,i}}.
\]
Substituting the above value into \eqref{logw}, we obtain
\[
1+\frac{1-q_{n,i}}{q_{n,i}}\,l_{n,i}^{\mathrm K}
=
1+\frac{p_{n,i}-q_{n,i}}{q_{n,i}}
=
\frac{p_{n,i}}{q_{n,i}},
\]
and
\[
1-l_{n,i}^{\mathrm K}
=
1-\frac{p_{n,i}-q_{n,i}}{1-q_{n,i}}
=
\frac{1-p_{n,i}}{1-q_{n,i}}.
\]
Hence
\[
\kappa_{n,i}
=
K_{n,i}(l_{n,i}^{\mathrm K})
=
p_{n,i}\ln\frac{p_{n,i}}{q_{n,i}}
+
(1-p_{n,i})\ln\frac{1-p_{n,i}}{1-q_{n,i}}
=
D_{\mathrm{KL}}(\mathrm{Ber}(p_{n,i})\|\mathrm{Ber}(q_{n,i})).
\]
If \(p_{n,i}\le q_{n,i}\), then the maximum of the function \(K_{n,i}\) on the interval
\([0,1)\) is attained for \(l=0\), and
\[
\kappa_{n,i}=K_{n,i}(0)=0.
\]
Finally,
\begin{equation}\label{KL}
\kappa_{n,i}
=
\begin{cases}
D_{\mathrm{KL}}
\left(
\operatorname{Ber}(p_{n,i})
\,\middle\|\,
\operatorname{Ber}(q_{n,i})
\right),
& \text{if } p_{n,i}\geq q_{n,i},\\[1ex]
0,
& \text{if } p_{n,i}<q_{n,i}.
\end{cases}
\end{equation}
for every \(n\) and \(i\). 

Under the restriction to nonnegative stakes, the local Kelly value
coincides with the Kullback--Leibler divergence when
$p_{n,i}\geq q_{n,i}$, with both quantities equal to zero at the
boundary $p_{n,i}=q_{n,i}$. If $p_{n,i}<q_{n,i}$, the restriction to
nonnegative stakes makes the optimal Kelly stake, and consequently
the local Kelly value, equal to zero, whereas the corresponding
Kullback--Leibler divergence is strictly positive.

In other words, under the restriction to nonnegative stakes, if
\(p_{n,i}>q_{n,i}\), the local Kelly value is equal to the Kullback--Leibler
divergence between the subjective distribution
\(\mathrm{Ber}(p_{n,i})\) and the market distribution \(\mathrm{Ber}(q_{n,i})\). The condition \(p_{n,i}>q_{n,i}\) means that the subjective probability
that the \(i\)-th horse wins is greater than the probability implied
by the market. In the player's assessment, this horse is therefore undervalued by the market, because
the market-implied probability is lower than the player's subjective
assessment of the probability of victory. In this sense, a bet on this horse's victory has, from the player's perspective, a positive edge relative to the market valuation. The optimal Kelly stake is
positive.

The local value of the entire race takes the form
\[
X_n=\max_{1\le i\le m_n}\kappa_{n,i}
=
\max_{1\le i\le m_n}
\Bigl(
\mathbf 1_{\{p_{n,i}>q_{n,i}\}}
\,D_{\mathrm{KL}}(\mathrm{Ber}(p_{n,i})\|\mathrm{Ber}(q_{n,i}))
\Bigr).
\]
Thus \(X_n\) is the largest exploitable informational value available in the \(n\)-th race
within the model of nonnegative stakes.
For horses satisfying \(p_{n,i}>q_{n,i}\), this value is equal to the corresponding
Kullback--Leibler divergence.

In the bookmaker model under consideration, the local value of a single bet $\kappa_{n,i}$ has a strict informational interpretation. It measures the exploitable divergence, under the restriction to nonnegative stakes,
between the subjective assessment of the probability of victory
and the probability implied by the market. The larger the value of \(\kappa_{n,i}\), the greater is, from the player's perspective,
the informational advantage relative to the market valuation with respect to
the given horse.

\subsection{The threshold rule as a filter of informational advantage}

In representation \eqref{KL}, the threshold rule
\[
\tau_r=\inf\{n\ge 1:\ X_n\ge r\}\,,
\]
means that the player accepts the first race in which the largest exploitable
informational value exceeds the level \(r\). In other words, the investor places a bet
only when the current informational advantage over the market is sufficiently large.

The numerator of the function \(\operatorname{rap}(r)\) measures the expected value
of the accepted informational advantage, taking into account the probability
that an opportunity exceeding the threshold \(r\) appears.

In the above context, the threshold rule can be interpreted as an informational filter accepting only those opportunities for which the divergence between the investor's model
and the market is sufficiently large. As a result, the threshold \(r^\star\) plays the role of
the minimal level of informational advantage required to place a bet.
Referring to the multi-armed model, to each arm \(i\) one can assign its own
random variable \(X^{(i)}\), interpreted as the local informational divergence
generated by a given source of market opportunities. The corresponding reservation index
\(r_i^\star\) then measures the maximal long-run informational advantage
per cycle attainable from this source. Choosing the arm with the largest \(r_i^\star\)
therefore means choosing the source that generates the highest effective
informational advantage after taking waiting time into account.

It is worth noting that the Kullback--Leibler divergence suggests a natural extension of the horse-racing
example presented above. If, in a single race, capital allocation
among all horses is allowed, rather than only the choice of one of them, then the divergence
\[
D_{\mathrm{KL}}(p_n\|q_n)
=
\sum_{i=1}^{m_n} p_{n,i}\ln\frac{p_{n,i}}{q_{n,i}},
\]
where
\[
p_n=(p_{n,1},\dots,p_{n,m_n}),
\qquad
q_n=(q_{n,1},\dots,q_{n,m_n}),
\]
constitutes a natural measure of divergence
between the investor's beliefs and the valuation made by the market.

\section{Conclusion}\label{last}

In this article, we have presented a simple analogue of the classical Gittins index approach for a one-armed decision problem with a continuous spectrum of payoffs. Unlike the original problem, the proposed model is transparent, which means that the decision mechanism is explicit. This model depends directly on the distribution of the observed variable \(X\) and has an immediate interpretation as the minimal value of an opportunity that justifies its acceptance. As in the Gittins problem, a complex problem of sequential choice is reduced to comparing a single number assigned to a given source of decisions. The index therefore plays the role of a reservation value, which makes it possible to order the available alternatives and then choose the one with the greatest current attractiveness.

The main result is to show that, in a model based on independent observations and the criterion of average payoff per cycle, the optimal decision has the natural form of a threshold rule, and the corresponding reservation threshold is determined by a fixed-point equation. This provides a well-defined and stable computational objective, thereby fitting into research on the search for simple, locally computed quantities that make it possible to plan sequential decisions. This is particularly important in situations where full dynamic programming is too complex or insufficiently transparent.


The basic results of the paper rely on the assumption that successive payoffs
are independent and have the same stationary distribution. Related simplifying
assumptions have a long tradition in classical sequential selection models
\cite{Gil,Fre,Fer}, which played an important role in the development of
optimal stopping theory and transparent stopping rules. The significance
of this assumption for the horse-racing example is discussed in
Section~\ref{kel12}. In a non-stationary environment or when successive observations are
dependent, a constant reservation threshold need not be optimal, and the
appropriate rule may depend on time, state, or the history of the process.

The model uses the criterion of expected average payoff per unit time and
therefore does not directly account for risk aversion, payoff variability,
capital constraints, or the risk of large losses. In the multi-armed version,
we also assume that sources of opportunities can be freely compared and
selected without switching costs, transaction costs, or liquidity constraints.
Including such elements could change both the value of the index and the form
of the optimal selection rule.

Another limitation is the exogenous nature of the payoff distributions. In the
model, decisions made by a participant do not affect future opportunities or
the behavior of other participants. The model therefore does not include
competition, imitation, learning by other participants, or strategic changes
in market conditions. In such situations, the value of a given source of
opportunities may depend on the common state of the market and the actions of
the other participants.

The learning procedure presented here reduces the problem to estimating a
single threshold, but a complete analysis of its convergence requires
additional conditions concerning the stability of the iterations, estimation
error, and the choice of learning steps. Extensions involving dependent or
non-stationary observations, switching costs, risk-sensitive criteria, and
strategic interactions may lead to thresholds depending on time and state and
may also require dynamic programming or numerical methods. The existence and
uniqueness of the optimal rule must then be examined separately.

It is also worth noting that the description of profits from bookmaker betting mentioned in the final sections of the article has rather problematic value for a gambler betting on horse races. This follows from the fact that knowledge of the probabilities of horses' victories is based not on statistical estimators, for which appropriate measurement data are unavailable, but on highly subjective speculations made by experts concerning the current chances of a horse in a particular race. This colorful model is nevertheless an interesting point of reference for a broad family of equivalent applications of prefix coding related to information redundancy \cite{Sha,Huf,Cov,Ris,Kri,Sht} (where data resources exist that are sufficiently rich for the required estimations). The topic indicated here, because of its breadth, requires treatment in a separate article.

\vspace{6pt}



\section*{Conflicts of Interest}
The authors declare no conflicts of interest.

\end{document}